\documentclass[12pt,reqno]{amsart}

\usepackage{amsmath}
\usepackage{amsfonts}
\usepackage{amssymb}
\usepackage{verbatim}
\usepackage[left=2.9cm,right=2.9cm,top=3cm,bottom=3cm]{geometry}
\usepackage{enumerate}

\usepackage{graphicx}

\newcommand{\calS}{{\mathcal S}}

\newcommand{\calG}{{\mathcal G}}

\newcommand{\supp}[1]{{supp\,(#1)}}

\newcommand{\calA}{{\mathcal A}}

\newcommand{\calC}{{\mathcal C}}

\newcommand{\N}{\mathbb N}

\newtheorem{theorem}{Theorem}
\newtheorem{lemma}{Lemma}
\newtheorem{rem}{Remark}

\newtheorem{example}{Example}
\newtheorem{definition}{Definition}

\title[Constrained Monte Carlo Markov Chains on Graphs]{Constrained Monte Carlo Markov Chains on Graphs}

\author{Roy Cerqueti}
\address{University of Macerata, Department of Economics and Law.
Via Crescimbeni 20, I-62100, Macerata, Italy}
\email{roy.cerqueti@unimc.it}

\author{Emilio De Santis}
\address{University of Rome La Sapienza, Department of Mathematics.
Piazzale Aldo Moro, 5, I-00185, Rome, Italy}
\email{desantis@mat.uniroma1.it}

\begin{document}

\begin{abstract}
This paper presents a novel theoretical Monte Carlo Markov chain procedure in the framework of graphs. It specifically deals with the construction of a Markov chain whose empirical distribution converges to a given reference one. The Markov chain is constrained over an underlying graph, so that states are viewed as vertices and the transition between two states can have positive probability only in presence of an edge connecting them. The analysis is carried out on the basis of the relationship between the support of the target distribution and the connectedness of the graph. 
\end{abstract}

\maketitle

\section{Introduction}


Monte Carlo Markov Chain (MCMC) problems represent a challenging
research theme not only for their natural practical implications but
also for the related methodological advancements.

The idea of a MCMC problem is to build a reversible regular Markov
chain with a target stationary distribution (see e.g. \cite{handbook,
fearn1, tierney}). To pursue this scope, several algorithms have
been proposed in the literature. Some of them are worthy to be
mentioned. 

In the Metropolis Hastings algorithm (see
\cite{hastings,metropolis}), a transition kernel is employed to
iteratively generate a value $y$ at time $t+1$ on the basis of the value $x$ observed at time $t$.

When the states space is huge the Metropolis Hastings algorithm must
be used with great care to avoid that the probabilities of
transition become too small and in practice unusable on the computer
simulation.


The Gibbs sampler, see \cite{geman1984}, solves
the problem of the huge cardinality in presence of a multivariate structure for the states space.
The strategy is to change state by changing only one of the
components of the multivariate state. In so doing, there are few
transition probabilities that are different from zero; therefore,
they remain not too small in order to be used on a computer. The
Gibbs sampler loses meaningfulness when the multivariate structure of
the state space is not identified.

The debate on the validity of the Gibbs sampler has been remarkably
enriched by \cite{green}. In the quoted paper, the Author elaborates
on \cite{tierney} and deals with a Bayesian choice of a vector of
models, whose individual components are selected among a set of
countable candidates. Each model have a number of unknown
parameters; such a number is not constant, and depends on the
considered component of the vector of models. In this context of not
fixed dimension of the parameter set, \cite{green} adapts to this context the
Metropolis-Hastings algorithm, by proposing a so-called "reversible
jump" version of it (see also \cite{scaccia} for further advancements). In \cite{carlinchib}, the Authors observe that
the convergence issues of the MCMC procedures arise always when the
problem involves the selection of one among a number of different
model specifications. To solve the convergence matter,
\cite{carlinchib} proposes a modified Gibbs sampler procedure
obtained by introducing a sort of average of the considered models.
In general, the issue of the convergence is a critical aspect, as also akcnowledged by Persi Diaconis in his long experience of scientific research and publications in the field. In this respect, we strongly recommend the reading of Diaconis' personal view on the matter, with some relevant insights of the future development of the MCMC in both areas of mathematical advancements and practical applications (see \cite{diac1, diac2}).

Our paper adds to this debate by dealing with a \emph{constrained} MCMC problem. In particular, we construct some Markov chains whose
empirical distributions converge to a target distribution as time
goes to infinity and which are constrained to move among the nodes
that are adjacent in an assigned graph.


To present the problem in a proper way, some notation is needed.
We will refer hereafter to a \emph{graph} $G =(\calS, E )$, being
$\calS$ the set collecting the nodes and $E$ the set of the edges. 
The nodes $ s, t \in \calS$ are declared \emph{adjacent} in $G$ if
$\{s,t\}\in E$ or $s=t$.

We now state a definition linking graphs and stochastic
processes.

\begin{definition}\label{dinamica}
We say that a stochastic process  $X =( X(t) : t \in \mathbb{N})$ on
$\calS$ is \emph{consistent} with the graph $G=(\calS, E) $ if, for each $t\in \mathbb{N}$,      $ X(t)$
and $ X(t+1) $ are adjacent in $G$ with probability one.
\end{definition}

Given two graphs $G=(\calS, E)$ and $G'=(\calS', E')$ we
say that $G'$ is a \emph{subgraph} of $G$ if $\calS'\subset \calS$
and $E'\subset E$, and we write $G' \subset G$.

A particular class of subgraphs will be of interest in the
following. Specifically, the subgraph $G' =(\calS' , E') \subset G=(\calS , E)  $ is said to be
an \emph{induced subgraph} of $G$ if $s,t \in \calS'$ and $\{ s,t \}
\in E$ imply $\{ s,t \} \in E' $. In this case we write $G' = G
[\calS']$ in order to stress the dependence on the set of nodes $\calS'$.

We notice that Definition \ref{dinamica} implies that if a process $X=( X(t) : t \in \mathbb{N})$ is
consistent with a graph $G $ then it is also consistent with any
graph $G' $   such that $ G \subset G'$.


From now we only consider  $| \calS | < \infty $ and consequently a finite graph $G=(\mathcal{S},E)$.
Given  a finite graph $G=(\mathcal{S},E)$ and a distribution
$\mu=(\mu(s):s \in \calS)$, we will provide in this paper an answer
to the following question:

\begin{itemize}
\item[\textbf{Q:}] \emph{Is it possible to construct a (not
necessarily homogeneous) Markov chain $X =( X(t) : t \in
\mathbb{N})$ which is consistent with $G$ and such that its
empirical distribution converges almost surely to $\mu$
as $t$ goes to infinity?}
\end{itemize}

More precisely we aim at constructing a reversible
Markov chain $X =( X(t) : t \in \mathbb{N})$  with the following
properties:
 $X$ is consistent with the graph $G$ and
 \begin{equation}\label{limite1bis}
\lim_{t \to \infty}  \frac{1}{t} \sum_{m=0}^{t-1} \mathbf{1}_{\{ X
(m)=s \}} = \mu(s)  , \qquad s \in \calS \qquad a.s..
\end{equation}

The motivations to pose question
 \textbf{Q} are basically three:
 \begin{itemize}
 \item[a)] we
  face the problem of the large cardinality
  of the states space by controlling the
  transitions among the states through
  the edges of a graph;
   \item[b)] we introduce a clear structure of the states space through the graph so that
   one can think to get some
  desired properties such as stochastic monotonicity or fast convergence;
  \item[c)]
    the introduction of a graph which constrains the positive transitions of the Markov chain describes several real-life evolution phenomena, where it is possible to move in a single step only from a state to an "adjacent one".
 \end{itemize}

In the following, we provide an answer to question \textbf{Q} in all
possible situations and we show that when $G= (\calS, E)$ is connected then it is
possible to construct such a (not necessarily time homogeneous) Markov chain.

\section{Main results}
For a target probability measure $(\mu(s): s \in \calS)$ and a graph $G$, all the possible situations, along with the related answers to question \textbf{Q}, can be distinguished in four cases:
\begin{itemize}
\item[$(i)$] If the distribution $\mu$ is concentrated on a unique $\bar s \in \calS$, i.e. $\mu=\delta_{\bar
s}$, then one can construct the constant Markov chain $X=(X(t):t \in
\mathbb{N})$ such that $X(t)=\bar s$, for each $t$. By Definition
\ref{dinamica} and the concept of adjacent states, one has that $X$
is consistent with $G$ and \eqref{limite1bis} is trivially satisfied.
\item[$(ii)$] If $G[\supp{\mu}]$ is not connected but $\supp{\mu}$ is contained in a connected
component of $G$, then one can construct a nonhomogeneous Markov
chain which is consistent with $G$ and fulfilling condition \eqref{limite1bis}
(see Theorem \ref{dinamic2} and Theorem \ref{thdin2} part c. below).
\item[$(iii)$] If $G[\supp{\mu}]$ is not connected and $\supp{\mu}$ is not contained in a unique connected component of $G$, then it does not exist a stochastic
process which is consistent with $G$ and fulfilling
\eqref{limite1bis} (see Theorem \ref{thdin2} part b. below).
\item[$(iv)$] If $G[\supp{\mu}]$ is connected, then one
can construct a homogeneous Markov chain consistent with $G$ which
satisfies \eqref{limite1bis} (see Theorem \ref{thdin2} part a. below).
\end{itemize}

We now deal with item  $(ii)$.

Notice that, in this case, there exists a connected
component of $\calS$, say $\hat \calS$, such that $\supp{\mu}
\subset \hat \calS$ and $\supp{\mu}
\neq \hat \calS $. Without loss of generality and to avoid
the introduction of further notation, we assume that $G$ is
connected and we identify  $\hat \calS $
with $ \calS$.

For a given distribution $\mu=(\mu(s):s \in \calS)$, let us define, in case $(ii)$, the non-empty set
$$
\calA_k=\left \{s \in \calS: \mu(s) <\frac{1}{k} \right \},\qquad k
\geq 2
$$
and let the distribution $\eta_k=(\eta_k(s):s \in \calS)$ be
$$
\eta_k (s) = \frac{1}{|\calA_k|} \mathbf{1}_{\{ s \in \calA_k \}},
\qquad s \in \calS,
$$
i.e. $\eta_k$ is the uniform distribution on $\calA_k$. We also
define the distribution $\mu_k=(\mu_k(s):s \in \calS)$ as
\begin{equation}
\label{mukka} \mu_k = \frac{1}{k} \eta_k + \frac{k-1}{k} \mu .
\end{equation}
Notice that
\begin{equation}
\label{mukka1}
|| \mu_k  - \mu ||_{TV} =  \frac{1}{k}     || \eta_k  - \mu ||_{TV}
\leq \frac{1}{k} ,
\end{equation}
where $  || \cdot ||_{TV} $ is the total variation norm (see e.g. \cite{lindvall}).

Let $N$ denote the cardinality of $\calS$. Since $G[\supp{\mu}]$ is not connected, then it contains at least two points. Since $\supp{\mu} \subset \calS$ and $\calS$ is connected, then $N \geq 3$.
  By construction, for any
 $$
k > \bar k: =  \left    \lceil       \frac{1}{  \min \{ \mu
(s) >0 : s \in \calS\}} \right  \rceil
$$
  and since $ N \geq 3$, one
has
\begin{equation}\label{31ott}
\mu_k (s) \geq \frac{1}{(N-1) k}, \qquad s \in \calS .
\end{equation}
Let us label the elements of $\calS=\{ s^1, \ldots , s^{N} \}$ such that
$$
\mu( s^1) \geq \mu( s^2)\geq \cdots \geq \mu( s^N) .
$$
According to definition \eqref{mukka}, for $k > \bar k$,  one also obtains
\begin{equation}\label{muk}
    \mu_k( s^1) \geq \mu_k( s^2)\geq \cdots \geq \mu_k( s^N) >0 .
\end{equation}

We construct the transition matrix $P^{(\mu_k,G)}= (p_{l, m}: l,m = 1,
\ldots , N)$ related to the distribution $\mu_k$ and to the graph
$G=(\calS,E)$. The dependence on $k$ of the elements of matrix $P^{(\mu_k,G)}$ is conveniently omitted. For each $l,m=1,\dots,N$,
\begin{equation}\label{ordinati}
    p_{l,m } = \left\{
\begin{array}{ll}
    p, & \hbox{if $l<m$ and $\{s^l,s^m\} \in E$;} \\
    \frac{\mu_k(s^m)}{\mu_k(s^l)}p, & \hbox{if $l>m$ and $\{s^l,s^m\} \in E$;} \\
  p_{l}, & \hbox{if $l=m$;} \\
    0, & \hbox{otherwise,} \\
\end{array}
\right.
\end{equation}
where
\begin{equation}\label{p11}
 p_{l} =1- p \left [ \sum_{m': m'> l } \mathbf{1}_{\{ \{s^l,s^{m'}\} \in E \}}
 + \sum_{m':m'<l}\frac{\mu_k({s^{m'}})}{\mu_k(s^l)}
\mathbf{1}_{\{ \{s^l,s^{m'}\} \in E \}}    \right ]
\end{equation}
and
\begin{equation}\label{p}
p= \min_{l=1, \ldots , N}\frac{1}{2\left(\sum_{m':m'>l}
\mathbf{1}_{\{ \{s^l,s^{m'}\} \in E \}}
+\sum_{m':m'<l}\frac{\mu_k(s^{m'})}{\mu_k(s^l)} \mathbf{1}_{\{
\{s^l,s^{m'}\} \in E \}} \right)}   .
\end{equation}
Notice that by definition $p\leq \frac{1}{2}$. In fact, since $G$ is connected, there exists
at least an edge $\{ s^1,s^{m}\} \in E$, with $m > 1$; thus the
denominator of \eqref{p} is at least equal to $  2$, when $l =1$. Clearly,
$P^{(\mu_k,G)}$ is a transition or stochastic matrix.

 Definition \eqref{ordinati} assures that the couple
$(\mu_k, P^{(\mu_k,G)})$ is reversible.  Moreover, $P^{(\mu_k,G)} $
is irreducible, since $G$ is connected; thus, $\mu_k$ is the unique
invariant distribution of $P^{(\mu_k,G)}$. The transition matrix
$P^{(\mu_k,G)} $ is also aperiodic since, by \eqref{p11} and \eqref{p}, $p_{l} \geq
\frac{1}{2}$ for $l=1, \ldots , N$.

We introduce the ergodic coefficient of Dobrushin  (see \cite{dobru} and  \cite{bremo} p. 235),
which is defined as
\begin{equation}\label{dobru}
    \delta(P)=1-\inf_{i,j =1, \dots, N} \sum_{h =1}^Np_{i,h}\wedge p_{j,h}
\end{equation}
where $P=(p_{i,j}:i,j =1, \dots, N)$ is a stochastic matrix.
%
\begin{lemma}
\label{lem:D} Given the transition matrix   $P^{(\mu_k,G)}$ on $\calS$ constructed above, with $N = |\calS| \geq 3$,
the Dobrushin's ergodic coefficient can be bounded from above as follows
$$
\delta((P^{(\mu_k,G)})^{N-1}) \leq 1- \left (    \frac{c_N}{k}
\right )^{N-1} , 
$$
for any $k > \bar k $, where  $c_N =\frac{1}{2(N-1)^2} $.
\end{lemma}
\begin{proof}
For $k > \bar k  $,
condition  $N \geq 3 $ and inequalities  \eqref{31ott} and \eqref{muk} provide
\begin{equation}\label{musumu}
    1 \leq \frac{\mu_k(s^{m})}{\mu_k(s^l)} \leq k(N-1), \qquad \text{
for } l>m.
\end{equation}
Thus, by 
\eqref{musumu}  one obtains
$p \geq \frac{c_N}{ k}$, for $k > \bar k $.
 Then one has that, if $p_{l,m}\not=0$,
\begin{equation}
\label{p2} p_{l,m} \geq  \frac{c_N}{k}  .
\end{equation}
For $k > \bar k $,
 since the graph $G$ is connected and $p_{l} \geq \frac{1}{2}$  for each $l = 1, \ldots , N$,
then \eqref{p2} gives that
$$
p_{l,m}^{(N-1)} \geq     \left (    \frac{c_N}{k}    \right )^{N-1}
, \qquad \forall \, l,m =1, \dots, N,
$$
where $p_{l,m}^{(N-1)}$ is the transition probability from $s^l$ to
$s^m $ in $(N-1)$ steps.
\newline
Then, by definition of the ergodic coefficient of Dobrushin in \eqref{dobru}, one has the thesis.
\end{proof}

Given an arbitrary distribution over $\calS$, namely $\lambda
=(\lambda(s): s\in \calS)$,  we construct a non-homogeneous Markov
chain $X=(X(t):t \in \mathbb{N})$ with $\lambda$ as initial
distribution.  The transition matrix of the Markov chain $X$ at time $t \in \mathbb{N}$
will be  denoted by $P(t)=(p_{i,j} (t):i,j=1, \dots,N)$. 

Let us consider
an  increasing sequence of times $(t_\ell:\ell \in
\mathbb{N})$, and let us define
\begin{equation}\label{Penne}
    P(t) = \sum_{k =\bar k +1}^{\infty} P^{(\mu_k,G)} \mathbf{1}_{\{ t \in [t_k
, t_{k+1} ) \}} .
\end{equation}

%

\begin{theorem}\label{dinamic2}
Consider a connected graph $G=(\calS,E) $ and a distribution
$\mu=(\mu(s): s\in \calS)$. Assume that $G[\supp{\mu}]$ is not connected but $\supp{\mu}$ is contained in a connected
component of $G$.
Any Markov chain $X =( X(t) : t \in
\mathbb{N})$ constructed above with transition matrix given in \eqref{Penne},
with sequence of times $(t_\ell   = \ell^{5N}     :\ell \in \N)$  is consistent with $G$ and  \eqref{limite1bis} holds true, i.e. 
$$
\lim_{t \to \infty}  \frac{1}{t} \sum_{m=0}^{t-1} \mathbf{1}_{\{ X
(m)=s \}} = \mu(s)  , \qquad s \in \calS \qquad a.s..
$$
\end{theorem}
\begin{proof}
The fact that $X$ is consistent with $G$ derives from the construction of $P$ (see \eqref{ordinati} and \eqref{Penne}).

To prove the result, we first check that 
\begin{equation}\label{limite222}
\lim_{\ell \to \infty}  \frac{1}{t_{\ell}} \sum_{m=0}^{t_{\ell}-1}
\mathbf{1}_{\{ X (m)=s \}} =\mu(s)  , \qquad s \in \calS \qquad
a.s..
\end{equation}
By definition of $(t_\ell:\ell \in \N)$ one
has
$$
\lim_{\ell \to \infty} \frac{t_{\ell+1}-t_\ell}{t_{\ell}} =0 \text{ and } \lim_{\ell \to \infty} \frac{t_{\ell+1}}{t_{\ell}} =1.
$$
Then \eqref{limite222} implies  \eqref{limite1bis}. In fact, for $t \in [t_\ell , t_{\ell+1} )$, one has
$$
\frac{1}{t_{\ell+1}} \sum_{m=0}^{t_\ell-1}  \mathbf{1}_{ \{ X(m) =s \}} \leq
\frac{1}{t} \sum_{m=0}^{t}  \mathbf{1}_{ \{ X(m) =s \}} \leq  \frac{1}{t_{\ell}} \sum_{m=0}^{t_{\ell+1}-1}  \mathbf{1}_{ \{ X(m) =s \}}
$$
$$
\leq \frac{1}{t_{\ell}} \sum_{m=0}^{t_\ell-1}  \mathbf{1}_{ \{ X(m) =s \}} +  \frac{t_{\ell+1}-t_\ell}{t_{\ell}}.
$$
Thus 
$$
\lim_{t \to \infty} \frac{1}{t} \sum_{m=0}^{t}  \mathbf{1}_{ \{ X(m) =s \}} = \lim_{\ell \to \infty } \frac{1}{t_{\ell}} \sum_{m=0}^{t_\ell-1}  \mathbf{1}_{ \{ X(m) =s \}}. 
$$

For $\varepsilon >0$ and $s \in \calS$ let us define the sequence of events $(
B_\ell (\varepsilon, s) : \ell \in \N)$ as
\begin{equation}\label{BBX}
  B_{\ell}  (\varepsilon, s) = \left \{ \Big |\mu (s) - \frac{1}{t_{\ell+1} - t_{\ell}}
  \sum_{m= t_\ell}^{t_{\ell+1} -1} \mathbf{1}_{ \{ X(m) =s \}}\Big  | < \varepsilon \right
  \} .
\end{equation}
To obtain \eqref{limite222} it is enough that, for each $
\varepsilon
>0$ and $s \in \calS$ one has
$$
\mathbb{P} \left  ( \liminf_{\ell \to \infty }   B_{\ell}
(\varepsilon, s)
\right ) =1 .
$$
Now, take the auxiliary sequence of independent random variables $ (Y(t):t \in
\N)$ with values on $\calS$ such that $Y(i)$ has distribution
$\mu_k$ if $i \in [t_k , t_{k+1})$ (see \eqref{mukka} for the
definition of $\mu_k$).

Notice that for each initial distribution $\vartheta$ on $\calS$,
 Lemma \ref{lem:D} and
Dobrushin's Theorem (see  e.g.  \cite{bremo}) give that
\begin{equation}\label{Dobr2}
    || \vartheta P(t_\ell)^{\ell^{2N}}-\mu_\ell ||_{TV}
    \leq \delta(P(t_\ell)^{N-1})^{\lfloor \frac{\ell^{2N}}{N-1}
    \rfloor} \leq \left( 1-\left(\frac{c_N}{\ell}\right)^{N-1}   \right )^{ \left \lfloor \frac{\ell^{2N}}{N-1}
 \right   \rfloor} \leq \exp \left (-c_N^{N-1} \left
\lfloor\frac{\ell^{N+1}}{N-1} \right \rfloor \right ),
\end{equation}
for any $\ell > \bar k  $.

Let $\hat c_N=c_N^{N-1}$. Given $i \geq 0$ and $k \geq 1 $, by
the maximal coupling (see \cite{lindvall}) and inequality \eqref{Dobr2} one can couple $X(
t_{\ell}+ k \ell^{2N} +i )$ with $ Y( t_{\ell}+ k \ell^{2N} +i ) $
so that
\begin{equation}\label{X=Y}
     \mathbb{P}
( X( t_{\ell}+ k \ell^{2N} +i ) \neq  Y( t_{\ell}+ k \ell^{2N} +i )
) \leq \exp \left (-\hat{c}_N \left \lfloor\frac{\ell^{N+1}}{N-1}
\right \rfloor \right ),
\end{equation}
when $ t_\ell +k \ell^{2N} +i < t_{\ell +1} $.

Let us define the sequence of events $( A_{\ell,i}: \ell \in \N , i
\in [0, \ell^{2N} ) )$ by

\begin{equation}\label{Aelle}
A_{\ell,i}=\left\{ X(t_{\ell}+ a \ell^{2N} +i) =Y(t_{\ell}+ a
\ell^{2N} +i):  a\geq 1\, \text{ and }  t_{\ell}+a \ell^{2N} +i
\leq t_{\ell+1}-1 \right\},
\end{equation}

for any $\ell \in \N$ and any integer $i\in [0, \ell^{2N} )$.

By subadditivity, one has
$$
\mathbb{P} ( A_{\ell,i})  \geq 1 - (\ell +1)^{5N} \exp \left
(-\hat{c}_N \left \lfloor\frac{\ell^{N+1}}{N-1} \right \rfloor
\right ).
$$
We also set $\hat A_\ell =\bigcap_{i =0}^{\ell^{2N}-1 } A_{\ell,i}
$. Then
$$
    \mathbb{P} ( \{X(t) = Y(t) : t \in  [ t_\ell + \ell^{2N}, t_{\ell+1 }
) \}) = \mathbb{P} ( \hat A_\ell  )   \geq
$$
\begin{equation}
 \geq 1 - (\ell +1)^{7N} \exp \left (-\hat{c}_N \left
\lfloor\frac{\ell^{N+1}}{N-1} \right \rfloor \right ).
\label{Acappuccio2}
\end{equation}
By \eqref{Acappuccio2} and the first Borel-Cantelli lemma, one has
that $\mathbb{P}(\liminf_{\ell \to \infty} \hat A_\ell) =1$.

Now, for $\varepsilon >0 $ and $s \in \calS$, let us define the
sequence of events $( \hat B_\ell  (\varepsilon ,s): \ell \in \N)$
as
\begin{equation}\label{BBY}
\hat   B_{\ell}     (\varepsilon,s ) = \left \{ \Big | \mu(s) -
\frac{1}{t_{\ell+1} - t_{\ell}}
  \sum_{m= t_\ell}^{t_{\ell+1} -1} \mathbf{1}_{ \{ Y(m) =s \}}\Big  | < \frac{\varepsilon}{2}\right
  \} .
\end{equation}
A straightforward calculation gives that
$$
\liminf_{\ell \to \infty } (\hat B_\ell    (\varepsilon ,s)    \cap
\hat A_\ell  ) \subset \liminf_{\ell \to \infty }  B_\ell
(\varepsilon , s) .
$$
Therefore to end the proof it is enough to show
\begin{equation}
\mathbb{P} (
\liminf_{\ell \to \infty } \hat B_\ell (\varepsilon, s ) ) =1.
\label{nuovaf}
\end{equation}
Such a result is a consequence of  the convergence $\mu_\ell \to \mu $, as $\ell \to \infty $,   the large deviation bounds for i.i.d.
Bernoulli random variables and the first Borel-Cantelli lemma. This
concludes the proof.
\end{proof}

\begin{rem} \label{remgenerale}
The definition of $(t_\ell: \ell \in \N)$ provided in Theorem \ref{dinamic2}
represents only one of the possible choices. In this respect, it is
interesting to note that the proof of Theorem \ref{dinamic2} can be
adapted to other sequences $(t_\ell: \ell \in \N)$. For example, one
can take $t_{\ell+1}-t_\ell \geq c \ell^{5N-1}$, with $c >0$. In
this case, for any $\ell \in \N$, there exists $I_\ell \in \N$ and an increasing sequence
$$
t_\ell^{(0)} ,  t_\ell^{(1)} , \ldots , t_\ell^{(I_\ell)}
$$
such that $t_\ell=t_\ell^{(0)} $,  $t_\ell^{(I_\ell)}=t_{\ell+1}$
and the following property holds
$$
\lim_{\ell \to \infty} \sup_{i \in \{0,1,\dots, I_\ell-1\}}
\frac{t_\ell^{(i+1)}-t_\ell^{(i)}}{t_\ell^{(i)}} =0;\qquad
\lim_{\ell \to \infty}
\frac{t_\ell^{(0)}-t_{\ell-1}^{(I_\ell-1)}}{t_{\ell-1}^{(I_\ell-1)}}=0.
$$
By reproducing the arguments of the proof of Theorem \ref{dinamic2}  for the  sequence $
t_\ell^{(0)} ,  t_\ell^{(1)} , \ldots , t_\ell^{(I_\ell)}
$, one obtains
that the Markov chain on $\calS$ with an arbitrary initial distribution and transition matrix as in  \eqref{Penne}  satisfies \eqref{limite1bis}.
\end{rem}
Next example shows that the convergence of the distribution $\mu_k$
to the distribution $\mu$ should not be taken  \emph{too fast} and
$t_{\ell+1}-t_\ell$ should be not taken \emph{too small} in order to have
\eqref{limite1bis}. 
\begin{example}\label{controes}
Let us consider a graph $G=(\calS,E)$ with
$\calS=\{s^1,s^2,s^3,s^4\}$ and
$E=\{\{s^1,s^3\},\{s^3,s^4\},\{s^2,s^4\}\}$.

Let us take the distribution $\mu=(\mu(s):s \in \calS)$ having
$\mu(s^1)=\mu(s^2)=\frac{1}{2}$, and  define $t_\ell=\ell$, for each $\ell \in \N$, and
the sequence of distributions
$(\hat \mu_\ell  : \ell \in \N)$ where
$\hat \mu_\ell=\mu_{2^\ell} $ (see the definition in \eqref{mukka}).
In particular, $||\hat \mu_\ell-\mu||_{TV} \leq \frac{1}{2^\ell}$.

We take a non-homogeneous Markov chain $X=(X(t):t \in
\N)$ with transition matrix $P(\ell) =(p_{m,n} (\ell) : m,n =1, 2,3,4 )$,  at time $\ell $, given by
$$
P(\ell)=P^{(\hat \mu_\ell, G)}, \qquad \ell \in \N.
$$

Accordingly to the definition of $p$ given in \eqref{p} and omitting the dependence of $p$ on the index $\ell$, one has
\begin{equation}\label{pbis}
p=\frac{1}{2^{\ell+1}},  \qquad \ell \in \N.
\end{equation}
Thus, \eqref{pbis} gives that
$p_{1,1}(\ell)=1-\frac{1}{2^{\ell+1}}$ at time $\ell$ (see \eqref{p11}). Therefore,
the Borel-Cantelli's Lemma guarantees that
\begin{equation}\label{pbis2}
|\{\ell \in \N : X(\ell)=s^1, X(\ell+1) \not= s^1\}|<\infty \qquad
a.s.,
\end{equation}
and therefore
\begin{equation}\label{pbis3}
\mathbb{P} (\bigcap_{s \in \calS }\{\lim_{t \to \infty}  \frac{1}{t} \sum_{m=0}^{t-1} \mathbf{1}_{\{ X
(m)=s \}} = \mu(s) \}) =0.
\end{equation}
In fact, formula \eqref{pbis2} allows to consider only $\omega \in \Omega$ such that condition $|\{\ell \in \N : X(\ell)=s^1, X(\ell+1) \not= s^1\}|<\infty$ is satisfied. If $X(\ell)=s^1$ for a finite number of $\ell$, then
$$\mathbb{P} (\lim_{t \to \infty}  \frac{1}{t} \sum_{m=0}^{t-1} \mathbf{1}_{\{ X
(m)=s^1 \}} = \frac{1}{2} ) =0;$$
if $X(\ell)=s^1$ for infinite values of $\ell$, then \eqref{pbis2} states that
$$\mathbb{P} (\lim_{t \to \infty}  \frac{1}{t} \sum_{m=0}^{t-1} \mathbf{1}_{\{ X
(m)=s^2\}} = \frac{1}{2} ) =0.$$
\end{example}


Notice that Example \ref{controes} gives a natural comparison
between our setting and the simulated annealing (see \cite{k1983}). In both cases the
hope is that the rate of convergence is fast but, if one
tries to have an \emph{excessively high} rate of convergence, it leads to local minima (case of simulated
annealing) or not convergence of the empirical measure to the target distribution $\mu$ in  our framework. In the case of excessively fast convergence rate, the response to question
\textbf{Q} might be wrong, even if the Markov chain is consistent with the graph $G$.

Next result provides an answer to \textbf{Q} for items $(iii)$ and $(iv)$.

\begin{theorem}\label{thdin2}
The following three  sentences hold true:
\begin{itemize}
    \item[a.] if $G[\supp{\mu}]$ is connected, then each  homogeneous Markov chain
$X=(X(t) : t \in \N)$ with state space $\supp{ \mu }$ having
transition matrix equal to $P^{(\mu , G[\supp{\mu}])}$ defined in
\eqref{ordinati} 
satisfies \eqref{limite1bis}. Furthermore, $X $ is consistent with
$G$;
\item[b.] if $G[\supp{\mu}]$ is not connected and $\supp{\mu}$ is not contained in a connected component of $G$, then it does not exist a stochastic process consistent with $G$ which satisfies \eqref{limite1bis}.
\item[c.] if $G[\supp{\mu}]$ is not connected and $\supp{\mu}$ is contained in a connected component of $G$,
    then each homogeneous Markov
chain consistent with $G$ does not  satisfy \eqref{limite1bis};

\end{itemize}
\end{theorem}
\begin{proof} We prove a. Since $G[\supp{\mu}]$ is connected, then the transition matrix $P^{(\mu ,
G[\supp{\mu}])}$ is well defined. Moreover, $\mu $ is the unique
invariant distribution of $ P^{(\mu , G[\supp{\mu}])} $ because $ P^{(\mu , G[\supp{\mu}])} $  is
irreducible. Now, by applying the ergodic theorem, one has
\eqref{limite1bis}. The consistence of $X$ with $G$ follows from the
fact that, for $l \not= m$, $p_{l,m}
>0 $ implies $\{s^l,s^m\} \in E $.

We prove b. by contradiction. Assume that \eqref{limite1bis}
holds true for a stochastic process $(X(t) :t \in \N)$ which is consistent with $G$. Then for each $s \in \supp {\mu }$ one should have
\begin{equation}\label{infinitamente}
    \mathbb{P} ( \{ X(t ) =s, \,\,\, i.o.\} ) =1 .
\end{equation}
Let us consider $ s' , s'' \in \supp{\mu} $ which belong to two
different connected components of $G$. By
\eqref{infinitamente}, it  follows that $\mathbb{P} (T < \infty) =1$
where
$$
T :=T_{s'} \wedge T_{s''},
$$
with
$$
T_s : = \inf \{ t\in \N : X (t) = s\}, \qquad s \in \calS.
$$
 Without loss
of generality one can assume that $\mathbb{P} (T = T_{s'}) >0$.
Then, by the consistence of $X $ with the graph $G$, one has that
$$
\mathbb{P} (  \{t \in \N: X(t ) =s'' \} = \emptyset | X (T) = s' )
=1.
$$
 Therefore
 $$
 \mathbb{P} ( \{ X(t ) =s'', \,\,\, i.o.\} ) <1 ,
 $$
and this contradicts \eqref{infinitamente}.

Now, we prove c. Without loss of generality we can consider that the graph $G$ is connected, thus the connected component containing $\supp{\mu}$ is the whole space $\calS $.  Now,  we can reduce to the case of irreducible Markov chains. Indeed, if a Markov chain is not irreducible, \eqref{limite1bis} cannot be true, because the limit in formula \eqref{limite1bis}, admitting that it exists,  depends on the initial state of the Markov chain.

By hypothesis, there exist  two connected components of $G[\supp{\mu}]$, say $G[A]$ and $G[B]$, with $A,B \subset \supp{\mu}$ and $A \cap B=\emptyset$, and a path $\gamma =(s^{j_1},s^{j_2}, \ldots , s^{j_n} )$ of $G$ such that the transition matrix $P =(p_{i,j} : i, j =1, \ldots , N )$ has
$p_{j_{r} , j_{r+1} }>0 $, for $r=1, \ldots, n-1$ and $s^{j_a} \in A$, $s^{j_{a+1}} \notin \supp{\mu}$ and $s^{j_{a+h}} \in B$, for some $a=1, \dots, n-2$ and $h=2, \dots, n-a$.

Assuming that the homogeneous Markov chain $X$ satisfies \eqref{limite1bis},
we proceed by contradiction.
 Since the Markov chain is irreducible, then ergodic theorem guarantees that
\begin{equation}
\label{erge1}
\lim_{N \to \infty} \frac{|\{t \leq N : X(t)=s^{j_{a}} \}|}{N}
\end{equation}
does exist almost surely. Moreover, by  \eqref{limite1bis}, one has
\begin{equation}
\label{erge1b}
\lim_{N \to \infty} \frac{|\{t \leq N : X(t)=s^{j_{a}} \}|}{N} = \mu(s^{j_{a}}) \text{ a.s.}
\end{equation}
Therefore, \eqref{erge1b} gives
$$
\mu(s^{j_{a+1}}) =\lim_{N \to \infty} \frac{|\{t \leq N : X(t)=s^{j_{a+1}} \}|}{N} \geq
$$
$$
\geq \lim_{N \to \infty} \frac{|\{t \leq N : X(t)=s^{j_{a}}, X(t+1)=s^{j_{a+1}} \}|}{N}=\mu(s^{j_{a}}) \cdot p_{j_{a} , j_{a+1} }>0 .
$$
This is a contradiction since $s^{j_{a+1}} \notin \supp{\mu}$.
\end{proof}

\begin{rem}\label{osserva}
Suppose that we are under hypothesis (ii) and
let us consider $\varepsilon >0 $ and a fixed $k \geq \lceil
\frac{1}{\varepsilon}\rceil$. Part a. of
Theorem \ref{thdin2}  states that it is possible to select a homogeneous Markov chain $X=(X(t) : t
\in \N)$ having transition matrix equal to  $P^{(\mu_{k} , G)}$ (see
\eqref{mukka} and \eqref{ordinati}), which
 satisfies
 \begin{equation}\label{ergo}
\lim_{t \to \infty}  \left | \frac{1}{t} \sum_{m=0}^{t-1}
\mathbf{1}_{\{ X (m)=s \}} - \mu_k(s)  \right | =0  ,
\qquad s \in \calS \qquad a.s..
\end{equation}
Then, \eqref{mukka1} and \eqref{ergo} gives that
\begin{equation}\label{ergono}
\lim_{t \to \infty}  \left | \frac{1}{t} \sum_{m=0}^{t-1}
\mathbf{1}_{\{ X (m)=s \}} - \mu(s)  \right | \leq \varepsilon  ,
\qquad s \in \calS \qquad a.s..
\end{equation}
  Furthermore, $X $ is
consistent with $G$.

Some consequences of Theorems \ref{dinamic2} and \ref{thdin2} arise.
Let us consider  $f : \calS \to \mathbb{R}$. 

Under condition of Theorem \ref{dinamic2} or of Theorem \ref{thdin2}
a. one obtains
\begin{equation}\label{ergo1}
\lim_{t \to \infty }  \frac{1}{t} \sum_{m=0}^{t-1} f( X (m)) =
\mathbb{E}_\mu (f), \qquad a.s.,
\end{equation}
where $\mathbb{E}_\mu$ is the expected value with respect
to the distribution $\mu$, i.e.
$$
\mathbb{E}_\mu (f)=\sum_{s \in \calS}f(s)\mu(s).
$$
Moreover, when \eqref{ergo} holds true, then
\begin{equation}\label{ergo2}
    \lim_{t \to \infty } \left | \frac{1}{t} \sum_{m=0}^{t-1} f( X (m))
- \mathbb{E}_\mu (f) \right | \leq  \varepsilon     \max_{s \in \calS} |f(s)|       , \qquad a.s..
\end{equation}
Thus, accepting the  error $  \varepsilon     \max_{s \in \calS} |f(s)|   $ given in \eqref{ergo2}, that can be taken arbitrarily small, one can always use an homogeneous Markov chain to numerically compute  $\mathbb{E}_\mu (f)$.
\end{rem}

\subsection{A remark on suitable criteria for graph selection}
We point out that a proper selection of the graph may 
lead to a more efficient MCMC procedure. In particular, graphs can contribute to the reduction of the number of possible transitions among states, as also Gibbs sampler proposes (see e.g. \cite{geman1984}). In fact, when the number of the states is extremely large, then the unconstrained transition probabilities involving all the pairs of states may be too small, hence too difficult to simulate.
   In this respect, a proper choice of the graph should ensure the connections among highly probable states,
thus avoiding the creation of metastable states (sometimes called \textit{wells}, see \cite{Landim22, LLM2018}). Indeed, wells are states in which the Markov chain is expected to spend an extremely long time before being able to visit other high-probability ones. This would increase dramatically the mixing time and the convergence speed of the MCMC algorithm (see e.g. \cite{Frigessi2, Frigessi3, Frigessi1}).

In this context, a very useful reading are \cite{Daley68, FK13, PW96}, where the (stochastically) monotone MCMC is explored. 
In details, a Markov chain is said to be stochastically monotone when the states space is endowed with a partial order and there exists a coupling of the chain with itself that maintains the partial order of the states space at any time.
Stochastically monotone Markov chains are particularly simple in the simulation procedures (see \cite{FK13} and \cite{PW96} for connections with the perfect simulation literature). 
Now, let us assume that the states space $\calS $ is endowed with a partial order and consider the target distribution $\mu$ on $\calS$. Naturally, there are infinite Markov chains satisfying \eqref{limite1bis}. Some of them might be stochastically monotone, i.e. simple in the simulation process. The role of the graph in obtaining stochastically monotone Markov chains might then be crucial.

As a paradigmatic example, we can take the classical ferromagnetic Ising model assigning a spin $\sigma(i) \in \{-1,+1\}$ to each vertex $i \in V$ and assume that the set $\calS=\{-1,+1\}^V$ is endowed with a partial order such that $\sigma' \preceq \sigma''$ if and only if $\sigma'(i) \leq \sigma''(i)$ for each $i \in V$. In this situation, we have that the Markov chain identified by the Gibbs sampler is stochastically monotone, and this property leads to affordable simulation exercises for the convergence towards the Gibbs measure of the ferromagnetic Ising model (see \cite{PW96} and, more recently, \cite{DL2012}). There are also other Markov chains converging to the Gibbs measure which do not maintain the ordering of the states space (see e.g. \cite{cd18, PW96}).

It is not difficult to construct other examples for non-ferromagnetic Ising models (where the Gibbs sampler is not stochastically monotone) such that Markov chains consistent with suitably defined graphs are stochastically monotone.

\section*{Product graphs and product distributions}

We now introduce the standard definition of product of graphs, as in \cite{sabidussi}.  It leads to a simplification of the MCMC simulations. 


\begin{definition}     \label{decomponibile}
Consider two graphs $G_1= (\calS_{1},E_1), G_2=
(\calS_{2},E_2)$.  The strong product $G_1\boxtimes G_2$ is a graph $G= (\calS , E) $, where $\calS = \calS_1 \times  \calS_2 $ and $E$ collects the couples $\{(s_{1} , s_{2} )  ,  (\bar s_{1} , \bar s_{2} )\} $, with $(s_{1} , s_{2} )  ,  (\bar s_{1} , \bar s_{2} ) \in \calS$,  such that one of the following condition is verified
\begin{itemize}
\item $\{s_1, \bar{s}_1\} \in E_1$ and $s_2= \bar{s}_2$;
\item $s_1=\bar{s}_1$ and $\{s_2, \bar{s}_2\} \in E_2$;
\item $\{s_1, \bar{s}_1\} \in E_1$ and $\{s_2, \bar{s}_2\} \in E_2$.
\end{itemize}

\end{definition}

Since the strong product of graphs is associative (see \cite{sabidussi}), then Definition \ref{decomponibile} can be extended to any collection of $r>2$ graphs obtaining $G=G_1 \boxtimes \dots \boxtimes G_r$.


Let us consider now $r$ finite sets $\calS_1 , \dots , \calS_r$ and take a product distribution $\mu=\prod_{h=1}^r \mu_{h}$, where
$\mu_{h}$ is a distribution on the space $\calS_{h}$.
We construct $r$ independent Markov chains
$X_{1}=(X_{1}(t):t \in \N), \dots, X_{r}=(X_{r}(t):t \in
\N)$ such that the $h$-th Markov chain $X_{h}$ has state space
$\calS_{h}$ and an arbitrary initial distribution
$\lambda_{h}=(\lambda_{h}(s_{h}): s_{h} \in \calS_{h})$, for each $h=1, \ldots,r$.

Moreover, by replacing $\calS$ with $\calS_{h}$ and
$\mu$ with $\mu_{h}$, we replicate the construction  provided
before Theorem \ref{dinamic2}. In so doing, we take $k \in \N$
to define the distribution 
$$(\mu_{h})_k :=((\mu_{h})_k(s_{h}):
s_{h} \in \calS_{h}).
$$

Now, take a sequence of increasing times $(t_\ell^{(h)}:
\ell \in \N)$, such that
\begin{equation}\label{minimo}
\min_{h = 1, \ldots ,
r}t^{(h)}_{\ell+1}-t^{(h)}_\ell \geq c \ell^{5N-1},
\end{equation}
with $c $ a positive constant.

 The transition matrices of $X_{h}$ are $(P_{h}(t):t \in \N)$
as in \eqref{Penne}:
\begin{equation}\label{Penne2}
    P_{h}(t) = \sum_{k =1}^{\infty} P^{((\mu_{h})_k,G_h)} \mathbf{1}_{\{ t \in [t_k^{(h)}
, t_{k+1}^{(h)} ) \}}.
\end{equation}

We introduce the 
Markov chain 
\begin{equation} \label{Xdef}
X= \Big ( X(t) =(X_{1}(t), \dots, X_{r} (t))\in \calS : t\in \N \Big ) . 
\end{equation}

Next result is similar to Theorem \ref{dinamic2} but it
is based on the  independent Markov chains constructed above. 

\begin{theorem}\label{dinamic3}
Let $\calS = \prod_{h =1}^r  \calS_{h}$ and $G (\calS, E)=G_1(\calS_1, E_1)        \boxtimes  \cdots  \boxtimes G_r(\calS_r, E_r)  $. 
Let us consider a
product distribution $\mu=\prod_{h=1}^r \mu_{h}$ and consider the
Markov chains $X$ of \eqref{Xdef}. 

Then
\begin{equation}\label{XX11}
    \lim_{t \to \infty}\frac{1}{t}\sum_{m=0}^{t-1}
    \mathbf{1}_{\{X(m)=s\}}=\lim_{t \to \infty}\frac{1}{t}\sum_{m=0}^{t-1} \prod_{h=1}^r
    \mathbf{1}_{\{X_{h}(m)=s_{h}\}}=\prod_{h=1}^r
    \mu_{h}(s_{h})=\mu(s),
\end{equation}
for each $s=(s_{1}, \dots, s_{r}) \in \calS$.
\end{theorem}
\begin{proof}
By \eqref{minimo}  follows that
$$
\lim_{t \to \infty} \frac{   \left   | [0,t] \cap \left (
\bigcup_{h=1}^r \bigcup_{\ell=1}^\infty [   t^{(h)}_\ell ,
t^{(h)}_\ell + \ell^{2N}) \right )   \right       |}{t}   =0.
$$
In fact, for each $h = 1, \ldots , r$,
$$
\lim_{t \to \infty} \frac{\left | [0,t] \cap \left (
\bigcup_{\ell=1}^\infty [   t^{(h)}_\ell , t^{(h)}_\ell +
\ell^{2N}) \right ) \right |}{t}   =0,
$$
since
$$
\lim_{\ell \to \infty}\frac{\ell^{2N}}{t^{(h)}_{\ell + 1}
-t^{(h)}_{\ell } }\leq \lim_{\ell \to \infty}\frac{\ell^{2N}}{c
\ell^{5N -1} } =0.
$$
Thus, the times in $\bigcup_{h=1}^r \bigcup_{\ell=1}^\infty [
t^{(h)}_\ell , t^{(h)}_\ell + \ell^{2N}) $ can be neglected in
the procedure of checking \eqref{XX11}, i.e.
$$
    \lim_{t \to \infty}\frac{1}{t}\sum_{m=0}^{t-1}
    \mathbf{1}_{\{X(m)=s\}}=
    \lim_{t \to \infty}\frac{1}{t}\sum_{m=0}^{t-1}
    \mathbf{1}_{\{X(m)=s\}}\cdot \mathbf{1}_{\{m \notin
\bigcup_{h=1}^r \bigcup_{\ell=1}^\infty [   t^{(h)}_\ell ,
t^{(h)}_\ell + \ell^{2N})  \}}
$$
and also
$$
    \lim_{t \to \infty}\frac{1}{t}\sum_{m=0}^{t-1}
    \mathbf{1}_{\{X(m)=s\}}=
    \lim_{t \to \infty}\frac{1}{t}\sum_{m=0}^{t-1}\left[
    \mathbf{1}_{\{X(m)=s\}}+ \mathbf{1}_{\{m \in
\bigcup_{h=1}^r \bigcup_{\ell=1}^\infty [   t^{(h)}_\ell ,
t^{(h)}_\ell + \ell^{2N})  \}}\right].
$$
Let us define the set of times $A =\bigcup_{h =1}^r
\bigcup_{\ell=1}^\infty   [ t^{(h)}_\ell , t^{(h)}_\ell +
\ell^{2N})$.
Now we introduce the independent  random variables
$( Y_{h}  ( t  )  :  t \in \N , h= 1, \ldots , r)$.  The random variables $( Y_{h}  ( t  )  :  t \in \N )$,  with  label $h $, take
value on $S_{h}$.  Moreover,  if
$    t \in[t^{(h)}_k , t^{(h)}_{k+1})$ then
   $Y_{h}  ( t  )  $  has distribution $\mu_{h,k}$.


We now adapt formula \eqref{X=Y} to the Markov chain $X_{h}$.
If $\bar{t} \notin A$ then for each $h = 1, \ldots , r $ there exists  $ \bar{\ell} _h $ such that $\bar{t} $ belong to $  [ t^{(h)}_{\bar{\ell}_h} , t^{(h)}_{\bar{\ell}_h+1}  )$.
In this case formula \eqref{X=Y} becomes
\begin{equation}\label{Acapp22}
    \mathbb{P} ( X_{h}(\bar t) = Y_{h}(\bar t) )
 \geq 1 -  \exp \left (-\hat{c}_N \left
\lfloor\frac{{\bar \ell_h}^{N+1}}{N-1} \right \rfloor \right ),
\end{equation}
where we recall that $\hat{c}_N=\frac{1}{[2(N-1)^2]^{N-1}}$.

Hence, for any $\bar t \notin A$ one has that there
exist $\bar \ell_1,\ldots, \bar \ell_r \in \N$ such that $\bar t \in
\bigcap_{h=1}^r[ t^{(h)}_{\bar\ell_h} + {\bar \ell_h}^{2N} ,
t^{(h)}_{\bar \ell_h +1} ) $. Therefore, using the independence of the random variables $Y$'s and the independence
of the Markov chains $X$'s, one has
\begin{equation}\label{Acapp23}
    \mathbb{P} ( (X_{1}(\bar t), \ldots,
    X_{r}(\bar t))= (Y_{1}(\bar t), \ldots,  Y_{r}(\bar t)) )
 \geq  1 - \sum_{h=1}^r \exp \left (-\hat{c}_N \left
\lfloor\frac{{\bar \ell_h}^{N+1}}{N-1} \right \rfloor \right ).
\end{equation}
For $\bar t \in
\bigcap_{h=1}^r[ t^{(h)}_{\bar\ell_h} + {\bar \ell_h}^{2N} ,
t^{(h)}_{\bar \ell_h +1} )
$, the distribution
of $(Y_{1}(\bar t), \ldots, Y_{r}(\bar t))$ coincides
with $\prod_{h=1}^r\mu_{h,\bar {\ell}_h}$.

Thus, we have
\begin{equation}\label{TV2}
\left|\left|\mu-\prod_{h=1}^r\mu_{h,\bar \ell _h}\right|\right|_{TV}\leq
\sum_{h=1}^r{\bar \ell_h}^{-1}.
\end{equation}
Notice that any $\bar \ell _h $ increases to infinity when $\bar t$ goes to infinity. Therefore, the left-hand side of \eqref{TV2} goes to zero
as $\bar t$ goes to infinity.
Inequalities \eqref{Acapp23} and \eqref{TV2} give
an upper bound for the distance in total variation between the law of $X (\bar t )$ and the distribution $\mu$.

Now, by following the arguments in the proof of Theorem \ref{dinamic2}, 
we obtain equation \eqref{XX11}.
\end{proof}

\section{Conclusions}

The paper adds to the MCMC literature. In particular, it deals with the existence and identification of a Markov chain which is constrained to move among adjacent nodes of a graph and whose empirical distribution coincides with a prefixed one. In so doing, we classify the cases in which such a Markov chain exists and, in case of existence, when it can be homogeneous or not.

The presence of assigned constraints let the paper be quite different with respect to the classical Metropolis-Hastings Markov chain methods. Indeed, one of the most relevant consequences of the graph-based constraint is the possibility of not having homogeneous Markov chains satisfying question \textbf{Q}, but only nonhomogeneous ones.

The problem is also extended to the particular case of strong products of graph, where also the given distributions are of product type. In this context, we give a result which allows researchers to study the convergence of one Markov chain with a large amount of states by using indipendent Markov chains with small state spaces -- hence reducing the computational complexity of related simulation models. 

Some suggestions on the speed of convergence are also provided. However, the detailed analysis of this important point may be the topic for future research.

\bibliographystyle{abbrv}

\bibliography{EmilioRoy}

\end{document}